\documentclass[12pt]{amsart}
 \usepackage{amssymb}
 \usepackage{lipsum}
\usepackage{tikz-cd}
\usepackage{amsthm}
 \usepackage[utf8]{inputenc}
\usepackage[english]{babel}
 \usepackage{amsmath}
 \usepackage{MnSymbol}
 \usepackage{microtype}
 
\newtheorem{thm}{Theorem}
\newtheorem{cor}[thm]{Corollary}
\newtheorem*{thm-non}{Theorem}
\newtheorem{lem}[thm]{Lemma}
\theoremstyle{definition}
\newtheorem{definition}[thm]{Definition}
 \newtheorem*{claim}{Claim}

\theoremstyle{remark}
\newtheorem{remark}[thm]{Remark}

\theoremstyle{Proposition}
\newtheorem{prop}[thm]{Proposition}
\newtheorem*{prop-non}{Proposition}

\usepackage{thmtools}
\usepackage{thm-restate}

\usepackage{hyperref}

\usepackage{cleveref}

\DeclareMathOperator{\Spec}{Spec}
\DeclareMathOperator{\fppf}{fppf}
\DeclareMathOperator{\Exc}{Exc}
\DeclareMathOperator{\Bl}{Bl}
\DeclareMathOperator{\W_2(k)}{W_2(k)}
\DeclareMathOperator{\V}{V}
\DeclareMathOperator{\Supp}{Supp}
\numberwithin{equation}{section}

\begin{document}

\newcommand{\nospaceperiod}{\makebox[0pt][l]{\,.}}
\newcommand{\nospacepunct}[1]{\makebox[0pt][l]{\,#1}}

\newcommand*{\isoarrow}[1]{\arrow[#1,"\rotatebox{90}{\(\sim\)}"
]}

\newcommand\blankpage{%
    \null
    \thispagestyle{empty}%
    \addtocounter{page}{-1}%
    \newpage}

%\tikzset{cong/.style={draw=none,edge node={node [sloped, allow upside down, auto=false]{$\cong$}}},
         %Isom/.style={draw=none,every to/.append style={edge node={node [sloped, allow upside down, auto=false]{$\cong$}}}}}
\newcommand{\dotr}[1]{#1^{\bullet}} 
\newcommand{\defeq}{\mathrel{\mathop:}=}
\newcommand{\eqdef}{\mathrel{\mathop=}:}
\newcommand{\C}{\mathbb{C}}
\newcommand{\Q}{\mathbb{Q}}
\newcommand{\Z}{\mathbb{Z}}
\newcommand{\R}{\mathbb{R}}
\newcommand{\mO}{\mathcal{O}}
\newcommand{\mL}{\mathcal{L}}

\bibliographystyle{alpha}

\setcounter{tocdepth}{1}

\title{On the Kodaira vanishing theorem for log del Pezzo surfaces in positive characteristic}
\title[On Kodaira vanishing for log del Pezzos in char $p>0$]{On the Kodaira vanishing theorem for log del Pezzo surfaces in positive characteristic}
\author{Emelie Arvidsson}
\address{EPFL SB MATH CAG
MA C3 615 (Batiment MA)
Station 8
CH-1015 Lausanne}
%\curraddr{Department of Mathematics and Statistics,
%Case Western Reserve University, Cleveland, Ohio 43403}
\email{emelie.arvidsson@epfl.ch}
\thanks{The author was supported by SNF Grant \#200021/169639.}

%\author{Author Two}
%\address{Mathematical Research Section, School of Mathematical Sciences,
%Australian National University, Canberra ACT 2601, Australia}

%    General info

%\date{February, 2020  }



\begin{abstract}
We investigate the vanishing of  $H^1(X,\mO_X(-D))$ for a big and nef $\Q$-Cartier $\Z$-divisor $D$ on a log del Pezzo surface $(X,\Delta )$ over an algebraically closed field of positive characteristic $p$.

%We investigate the vanishing of  $H^1(X,\mO_X(-D))$ for a big and nef $\Q$-Cartier $\Z$-divisor $D$ on a log del Pezzo surface $(X,\Delta )$ of fixed Gorenstein index $I$ over an algebraically closed field of positive characteristic $p$. In particular, we give an effective bound on the characteristic of the base field for which this vanishing may fail in terms of the index $I$. As a consequence, we prove the Kodaira vanishing theorem for (not necessarily Cartier) divisors on a canonical del Pezzo surface over an algebraically closed field of characteristic $p\geq 9221$. We also prove that for $p\geq 5$, Kodaira vanishing holds for a big and nef (not necessarily Cartier) $\Z$-divisor $D$  on a log del Pezzo surface provided that $D$ has sufficiently big volume in relation to the Gorenstein index of the surface. We also show that the existence of an integer $p_0\geq 2$ such that all log del Pezzo surfaces of Picard rank one in char $p>p_0$ admits a log resolution that lifts to $\W_2(k)$ implies Kodaira vanishing for all log del Pezzo surfaces in char $p>p_0.$ In particular, Kodaira vanishing on all log del Pezzo surfaces in char $p>5$ can be deduced from the ongoing classification of log del Pezzo surfaces of Picard rank one in char $p>5$ by J. Lacini. 
\end{abstract}

\maketitle
%\tableofcontents

\section*{Introduction}
It has long been known that the (Kodaira and Kawamata--Viehweg) vanishing theorems, so fundamental to birational geometry in characteristic zero, in general fail for surfaces in positive characteristic \cite{Ra}. Several people have investigated different classes of surfaces over perfect fields of positive characteristic for which the (Kodaira and Kawamata--Viehweg) vanishing theorems may or may not hold. The question is subtle; for example, every smooth rational surface over an algebraically closed field satisfies Kodaira vanishing, however over any algebraically closed field of positive characteristic, there are smooth rational surfaces that violate Kawamata--Viehweg vanishing \cite{CT}. On the other hand a smooth surface with ample anti-canonical bundle satisfies Kawamata--Viehweg vanishing theorem over any algebraically closed field \cite[Proposition A.1]{CT}. This however is no longer true if (klt) singularities are allowed and counterexamples (even to Kodaira vanishing) have been constructed over algebraically closed fields of characteristic $2$ \cite{cascini2} and $3$ \cite{Bernasconi}. \\

The situation simplifies in large characteristics. In \cite{CTW} the authors prove the existence of an integer $p_0$ such that over an algebraically closed field of characteristic $p>p_0$ every log del Pezzo surface satisfies Kawamata--Viehweg vanishing theorem. Finding an effective bound for this $p_0$ is a central open question in positive characteristic birational geometry. For example over an algebraically closed field of characteristic $p>\mathrm{max}\{5, p_0\}$ a three dimensional klt singularity is rational \cite[Theorem 1.1]{HW}, and threefolds satisfy a refined version of the basepoint free theorem \cite{bernbpf}[Theorem 1.1].

The construction of the integer $p_0$ in \cite{CTW} is implicit. During the course of the proof of \cite[Theorem 1.1]{CTW} the authors consider log del Pezzo surfaces 
belonging to a bounded family. The construction of the integer $p_0$ does in particular depend on this family in an implicit way.
It is therefore natural to ask if, we, by other methods, may describe vanishing theorems in large \emph{explicit} characteristics depending on some explicit numerical invariant of a bounded family %$(\mathcal{X}, \mathcal{B})\rightarrow T$ 
of log del Pezzo surfaces. In this direction, we show that Kodaira vanishing (for big and nef divisors) holds on a del Pezzo surface of bounded index $I$ in characteristic $p>p_0(I)$ where $p_0(I)$ is an explicit polynomial in the index $I$. We also determine an explicit bound for $p_0$ (for Kodaira vanishing) in terms of $\epsilon >0$ for an $\epsilon$-klt log del Pezzo surface. %In the proof of \cite{CTW}[Theorem 1.1] they distinguish between two cases: log del Pezzo surfaces which are epsilon klt and those which are not (for a fixed small enough epsilon). We make an explicit bound on $p_0$ (for Kodaira vanishing) in terms of this fixed epsilon for the first class of surfaces. 
 Our main results in this direction are the following:

\begin{restatable}{thmx}{largecharI}\label{largecharI}  Let $X$ be a projective klt surface over an algebraically closed field $k$ of characteristic $p>0$ with ample anti-canonical divisor $-K_X$ of Cartier index $I$. Let $D$ be a big and nef $\Q$-Cartier $\Z$-divisor on $X$. \begin{enumerate}
\item If $I=1$, i.e., if $-K_X$ is an ample Cartier divisor then:\\
\begin{itemize} 
\item  if $p\geq 9221$, then $H^1(X, \mO_X(-D))=0$ 
\item if $p\geq 5$, then $H^1(X, \mO_X(D))=0.$\\
\end{itemize}
\item If $I\geq 2$ and $p\geq (13-45I)^2(2I^3+4I^2+2I)+1$, then: 
$$H^1(X, \mO_X(D))=H^1(X, \mO_X(-D))=0.$$
\end{enumerate}
\end{restatable}

\begin{restatable}{thmx}{vaneps}\label{van eps} Let $0<\epsilon< 3^{-1/2}$ be a real number. Let $(X,\Delta )$ be a projective $\epsilon$-klt surface with $-(K_X+\Delta)$ ample over an algebraically closed field $k$ of characteristic $$p\geq  2^{\frac{128}{\epsilon^5}+2}\left(\frac{2}{\epsilon}\right)^{\frac{(128)^2}{\epsilon^{25}}}
 +1.$$%2^{\frac{128}{\epsilon^5}+2}(\frac{2}{\epsilon})^{\frac{(128)^2}{\epsilon^{25}}} +1. $$
Then Kodaira vanishing %(even for big and nef divisors) 
holds on $X$, i.e., for all ample $\Q$-Cartier $\Z$-divisors $D$ on $X$ we have that $H^i(X,\mO_X(D+K_X))=0$, for all $i>0$. \end{restatable}

In this note, we use a technique due to Ekedahl \cite{Ekedahlcanmod}. When trying to prove Kodaira vanishing in explicit characteristics using this technique, the difficulties arise from Weil divisors which are far from being Cartier. For example, on a log del Pezzo surface over an algebraically closed field of characteristic $p\geq 4c+1$ we can prove that $H^1(X,\mO_X(K_X+A))=0$  for every  ample $\Q$-Cartier Weil divisor $A$ of Cartier index $\leq c$ (see Remark \ref{Kodairavanforcartier}). Another strategy is therefore, loosely speaking, to, in different ways, control the Cartier index of a Weil divisor $D$ in terms of the Cartier index $I$ of $(K_X+\Delta )$. % and to apply the techniques developed in the paper \cite{Ekedahlcanmod}.   
We prove, in this way, that a big and nef $\Z$-divisors $D$ of \emph{big volume} relative to the index $I$ in characteristic $p\geq 5$ satisfies $H^1(X, \mO_X(-D))=0$.  %that $H^1(X,\mO_X(-D))=0$ for any ample $\Z$-divisor such that $D^2\geq f(I)$ where $f(I)$ is an explicit degree 3 polynomial in the index $I$ of $X$. \\

\begin{restatable}{thmx}{bigvolcor}
\label{thm:bigvolcor}Let $(X, \Delta )$ be a projective klt surface over an algebraically closed field $k$ of characteristic $p\geq 5$ with $-(K_X+\Delta )$ an ample $\Q$-Cartier $\Q$-divisor of Cartier index $I$. Let $D$ be a big and nef $\Q$-Cartier $\Z$-divisor. Then the following holds true:
 \begin{enumerate}
\item If $I=1$ and  $D^2\geq 9$, then $H^1(X, \mO_X(-D))=0.$
\item If $I\geq 2$ and $D^2\geq {2I^3+4I^2+2I}$, then $H^1(X, \mO_X(-D))=0.$
\end{enumerate}
\end{restatable}

During the preparation of this note we were informed that J. Lacini has classified all log del Pezzo surfaces of Picard rank one over an algebraically closed field of characteristic $p>5$ \cite[Theorem 1.1]{JustinLacini}. From this classification it follows that all log del Pezzo surfaces of Picard rank one over an algebraically closed field of characteristic $p>5$ admit a lift to characteristic zero over a smooth base \cite[Theorem 7.2]{JustinLacini}. It therefore follows from his work that Kodaira vanishing holds for log del Pezzo surfaces of Picard rank one over an algebraically closed field of characteristic $p>5$ (see \cite[Lemma 6.1]{CTW}).  Combining his result with the techniques exploited in this paper, one can prove the Kodaira vanishing theorem on a log del Pezzo surface over an algebraically closed field of characteristic $p>5$. %Therefore, knowing that a Picard rank one log del Pezzo surface over an algebraically closed field of characteristic $p>5$ admits a log resolution which lifts to characteristic zero over a smooth base, one can get a much better result than those discussed above. 
At the end of this note we illustrate this argument. \emph{Building on the liftability of rank one log del Pezzo surfaces in characteristic $p>5$ by J. Lacini we have:}

\begin{restatable}{thmx}{ratvannef}
 Let $(X,\Delta )$ be a log del Pezzo surface over an algebraically closed field of characteristic $p>5$. Then Kodaira vanishing holds on $X$, i.e., for all ample $\Z$-divisors $D$ on $X$ we have $H^i(X,\mO_X(D+K_X))=0$, for all $i>0$.\end{restatable}

We do not know if the results of J. Lacini may be extended to show that every log del Pezzo surface over an algebraically closed field of characteristic $p>5$ admits a log resolution that lifts to characteristic zero over a smooth base. In fact, we do not know if every log del Pezzo surfaces in \emph{large characteristic} admits such a lift (see also \cite[Theorem 1.1]{CTW}).

\begin{remark} We do not know if the Kodaira vanishing theorem holds on a log del Pezzo surfaces over an algebraically closed field of characteristic $p=5$.\end{remark}

\specialsection*{Acknowledgments}
I would like to thank my PhD-advisor Zsolt Patakfalvi for the help he has given me during this work. I would like to thank Jakub Witaszek for answering questions about \cite{CTW}. I also thank Fabio Bernasconi and Maciej Zdanowicz for very useful discussions and for reading earlier drafts of this paper.

\numberwithin{thm}{section}

\section{Preliminaries}
%\begin{enumerate}

%\item \label{1}

By a variety we shall mean a finite type integral separated scheme over a field. We will work exclusively over an algebraically closed field.\\

A $\mathbb{Q}$-divisor $D$ is said to be $\Q$-Cartier if there exists an integer $m$ such that $mD$ is Cartier. A $\mathbb{Q}$-divisor is ample/nef/big if it is $\Q$- Cartier and an integer multiple is ample/nef/big as a line bundle. %A $\mathbb{Q}$ (resp. $\R$)-divisor is ample/nef/big if it can be written as a positive $\Q$ (resp. $\R$) -linear combination of ample/nef/big $\Z$-divisors. 
\\

If $X$ is a normal variety, then the reflexive sheaves on $X$ are determined (up to isomorphism) by their restriction to a big open subset $U$ (an open subset $U$ of $X$ is big if $\mathrm{codim}_X (X-U)\geq 2$). % I if 
If $i: U \rightarrow X$ is the inclusion of a big open subset and $\mathcal{M}$ is a reflexive sheaf on $U$ then $i_*\mathcal{M}$ is a reflexive sheaf on $X$, moreover for any reflexive sheaf $\mathcal{G}$ on $X$ we have an isomorphism $\mathcal{G}\cong i_*\mathcal{G}_{|_U}$  \cite[Proposition 1.6]{refl}.
% and in particular $G$ is uniquely determined on a big open. 
Let $U$ denote the smooth locus of $X$. Since $X$ is normal $U$ is big. %If $F$ is a rank one reflexive sheaf on $X$ then its restriction to the smooth locus is necessarily locally free. 
Therefore, for any Weil divisor $D$ on $X$ the restriction of $D$ to the smooth locus defines a reflexive sheaf $i_*\mathcal{O}_U(U\cap D)$ of rank one on $X$. Conversely, any reflexive sheaf of rank one $\mathcal{F}$ on $X$ defines a Weil divisor $D$ by setting $D$ to be the closure of a Weil divisor $E$ on $U$ satisfying $\mathcal{F}_{|_U}=\mathcal{O}_U(E)$. This defines a one to one correspondence between the reflexive sheaves of rank one on $X$ up to isomorphism and the set of Weil divisors on $X$ up to rational equivalence \cite{refl}. We denote by $\mathcal{O}_X(D)$ the reflexive sheaf corresponding to a divisor $D$ on $X$. 
We use the notation $K_X$ for the class corresponding to the closure of a canonical divisor $K_U$ on $U$ in $X$, hence $\mathcal{O}_X(K_X)\cong i_*\mathcal{O}_U(K_U)$. If X is projective then $\mathcal{O}_X(K_X)$ can be seen to be a dualizing sheaf for $X$ \cite[Proposition 5.75]{KolMor}. \\

%\item \label{4} 

%\item \label{5} 
In general if $D_1$ and $D_2$ are Weil divisors on a normal variety $X$, then $\mathcal{O}_X(D_1)\otimes \mathcal{O}_X(D_2)$ is not reflexive. However, the dual $\mathcal{M}^v$ of a coherent sheaf $\mathcal{M}$ is always reflexive \cite{refl}[Corollary 1.2] and we may define a product structure on the set of reflexive sheaves by defining  $\mathcal{O}_X(D_1)\otimes \mathcal{O}_X(D_2)\defeq  (\mathcal{O}_X(D_1)\otimes \mathcal{O}_X(D_2))^{\vee \vee }$. We have the following identities for Weil divisors $D_1, D_2$ and a Cartier divisor $C$:
$\mathcal{O}_X(D_1+D_2)\cong (\mathcal{O}_X(D_1)\otimes \mathcal{O}_X(D_2))^{\vee \vee }$, $\mathcal{O}_X(-D_1)\cong \mathcal{O}_X(D_1)^{\vee}$, $\mathcal{O}_X(D_1+C)\cong \mathcal{O}_X(D_1)\otimes \mathcal{O}_X(C)$ and $Hom_{\mathcal{O}_X}( \mathcal{O}_X(D_1),  \mathcal{O}_X(D_2))\cong \mathcal{O}_X(D_2-D_1).$ %In particular,  $\mathcal{O}_X(mD_1)$ is isomorphic to the $m$-fold tensor product:  
%$$(\mathcal{O}_X(D_1)\otimes (\mathcal{O}_X(D_1) \otimes (\dots \mathcal{O}_X(D_1))^{\vee \vee }\dots)^{\vee \vee })^{\vee \vee } )^{\vee \vee }\cong (\mathcal{O}_X(D_1)^{\otimes m})^{\vee \vee }.$$
This product structure makes the set of reflexive sheaves on $X$ (up to isomorphism) into a group such that the correspondence between Weil divisors and reflexive sheaves on $X$, as described %in point \ref{Weil-reflexive} 
above, becomes an isomorphism of groups.\\ %If we 

By a log pair $(X, \Delta)$ we shall mean a normal variety $X$ together with an effective $\Q$-divisor $\Delta$ on $X$ such that $K_X+\Delta$ is $\Q$-Cartier. By a klt (respectively lc) pair we shall mean a log pair with klt (respectively lc) singularities in the sense of \cite[Section 2]{KolMor}. By \cite[Corollary 4.11]{Tan16} if $X$ is a surface and $(X, \Delta)$ is a klt pair then $X$ is $\Q$-factorial.\\

A normal surface $X$ is said to be a del Pezzo surface if $-K_X$ is $\Q$- Cartier and ample. A log pair $(X, B)$ is said to be a log del Pezzo surface if $-(K_X+ B)$ is ample. \\

%\end{enumerate}

\subsection{Serre vanishing and Serre duality for $\Z$-divisors}
One of the most fundamental vanishing theorems in algebraic geometry is that of Serre:
\begin{thm-non} Let $\mathcal{L}$ be an ample line bundle on a proper scheme $X$ and let $\mathcal{F}$ be a coherent sheaf on $X$. Then for all $m$ big enough we have: $H^i(X,\mathcal{F}\otimes \mathcal{L}^{\otimes m})=0,$ for all $i>0$.\end{thm-non}
It naturally implies a Weil divisor version:

\begin{cor}\label{serrevanQ} Let $X$ be a normal projective variety. If $D$ is an ample $\mathbb{Q}$-Cartier $\Z$-divisor and $\mathcal{F}$ is a coherent sheaf on $X$ then: $$H^i(X, \mathcal{F}\otimes \mathcal{O}_X(lD))=0$$ for all $l$ big enough and all $i>0$.\end{cor}
\begin{proof}  Let $m$ be an integer such that $mD$ is Cartier. For all $1\leq i\leq m$ there exists an $n_i$ such that $H^j(X, \mathcal{F}\otimes \mathcal{O}_X(iD)\otimes \mathcal{O}_X(mD)^{\otimes n})=0$, for all $n\geq n_i$ and all $j>0$. Since $\mathcal{O}_X(mD)$ is Cartier we have that $\mathcal{O}_X(iD)\otimes \mathcal{O}_X(mD)^{\otimes n}=\mathcal{O}_X((i+mn)D)$. Set $N=\max_{1\leq i \leq m} n_i$. For all $l\geq (m+1)N$, we can write $l=i+mk$, for $1\leq i\leq m-1$, where $k\geq N$. Therefore, we have $H^i(X, \mathcal{F}\otimes \mathcal{O}_X(lD))=H^i(X, \mathcal{F}\otimes \mathcal{O}_X(iD)\otimes \mathcal{O}_X(mD)^{\otimes k})=0$, for all $l\geq (m+1)N$.
\end{proof}

\begin{thm}[{\cite[Theorem 10]{fujita1983vanishing}}]\label{Sernefbig} Let $X$ be a normal projective surface and $L$ be a big and nef line bundle on $X$. Then there exists an integer $m$ such that:
$$H^1(X,\mO_X(-nL-E)=0$$for all nef and effective divisor $E$ and  for all $n>m$. \end{thm}

\begin{cor}\label{Sernefbigcor}  Let $D$ be a big and nef integral divisor on a normal projective surface $X$, then $H^1(X, \mO(-mD))=0$ for all $m>>0$.\end{cor}
\begin{proof}Let $r$ be the Cartier index of $D$. Let $n>0$ be such that $H^0(X,\mO(lD))$ is non zero for all $l\geq n$, this is possible since $D$ is big.  By Theorem \ref{Sernefbig} there exists an integer $m$ such that $H^1(X, \mO(-(kr+l)D))=0$ for all $l\geq n$ and all $k\geq m$, this proves the Corollary. \end{proof}

We will repeatedly need to use a form of Serre Duality valid for reflexive sheaves on a normal surface:

\begin{thm}[Serre Duality for CM-sheaves] \label{serdual} Let $X$ be a projective scheme of pure dimension $n$ over a field $k$. Let $\mathcal{F}$ be a CM sheaf on $X$ such that $\Supp(\mathcal{F})$ is of pure dimension n. Then 

\begin{center}$H^i(X,\mathcal{F})$ is dual to $H^{n-i}(X,Hom_{\mathcal{O}_X}(\mathcal{F},\omega_X))$.\end{center}
\end{thm}

\begin{proof}E.g., \cite[Theorem 5.71]{KolMor}\end{proof}

In particular, by \cite[Proposition 1.3]{refl} for any $\Z$-divisor $D$ on a normal projective surface $X$ we have that $H^i(X,\mathcal{O}_X(D))$ is dual to $H^{2-i}(X,\mathcal{O}_X(K_X-D))$ for $i=0,1$.

\begin{remark}If $D$ is a big and nef $\Z$-divisor on a normal projective surface $X$ then $H^0(X, \mO_X(-D))=0$ and hence by duality $H^2(X, \mO_X(D+K_X))=0$. Therefore, Kodaira vanishing is equivalent to the vanishing of $H^1(X,\mathcal{O}_X(-D))\cong H^{1}(X,\mathcal{O}_X(K_X-D)).$\end{remark}

\subsection{Frobenius techniques}\label{7}
Let $X$ be a scheme over a positive characteristic base $S$, let $F_X$ denote the absolute Frobenius on $X$. %If $X$ is a scheme over a scheme $S$ of characteristic p, then 
$F_X$ is not a morphism of $S$-schemes unless $F_S$ on $S$ is the identity. In particular we have the following commuting diagram:
$$ \begin{tikzcd}
X\arrow[bend left]{drr}{F_X}
\arrow[bend right]{ddr}
\arrow[dotted]{dr}[description]{F_{X/S}} & & \\
& X \times_S S \arrow{r}{\omega} \arrow{d} & X \arrow{d} \\
& S \arrow{r}{F_S} & S\nospaceperiod
\end{tikzcd}$$
where $F_{X/S}$ denotes the relative Frobenius. When $S=Spec(k)$ for a perfect field $k$ then the absolute Frobenius on $S$ is an isomorphism and \begin{tikzcd} X \times_S S \arrow{r}{\omega} & X \end{tikzcd} is an isomorphism \emph{over $S$} if we consider $X$ as a scheme over $S$ by post-composing the structure morphism to $S$ with $F^{-1}_S$ on $S$. We often do this identification, under which the relative and the absolute Frobenius coincide, and denote by $F: X\rightarrow X$ the corresponding morphism.
Let $X$ be a normal variety over an algebraically closed field of positive characteristic. If $D$ is an integral divisor on $X$ then $(F^*\mathcal{O}_X(D))^{\vee \vee }\cong \mathcal{O}_X(pD)$, since they agree on the regular locus. There is a natural map (unit of adjunction) $\mO_X(D)\rightarrow F_*F^*\mO_X(D)$. This induces a morphism: $$H^i(X, \mO_X(D))\to H^i(X, F^*\mO_X(D)) \rightarrow H^i(X, F^*\mO_X(D)^{\vee \vee }).$$ Therefore, the Frobenius on X induces morphisms for all $m$:
$$H^i(X,\mathcal{O}_X(-p^mD))\to H^i(X,\mathcal{O}_X(-p^{m+1}D)).$$

If $D$ is an ample CM divisor, then these groups eventually vanish for $m>>0$ by Theorem \ref{serdual} and Corollary \ref{serrevanQ}. Therefore, Kodaira vanishing theorem for ample CM $\Z$-divisors $D$ holds true on $X$ if and only if the morphisms $H^i(X,\mO
_X(-p^mD))\rightarrow H^i(X,\mO _X(-p^{m+1}D))$ are injective for all $m$ and all $i<\mathrm{dim}(X)$. %This motivates us to consider the following class of varieties:

\subsection{Non-vanishing and the associated $\alpha_{\mO_{X_0}(D_0)}$-torsor.}\label{alphasit}
Let $X$ be a projective normal surface and $D$ be a $\Q$-Cartier $\Z$-divisor on $X$. Let $X_0\subset X$ denote the smooth locus of $X$ and $D_0\defeq D\cap X_0$ be the restriction of $D$ to the smooth locus. Let $F_0$ denote the Frobenius on $X_0$. We will assume that $H^1(X, \mO_X(-D))\neq 0$ and that $H^1(X,\mO _X(-pD))=0$. Let $\alpha \in H^1(X, \mO_X(-D))$ denote a non-trivial element of the kernel of the map $H^i(X,\mathcal{O}_X(-D))\rightarrow H^i(X,\mathcal{O}_X(-pD))$ described in subsection \ref{7} above. Since $Z\defeq X-X_0$ is the complement of a big open subset of $X$ and $\mO_X(nD)$ is reflexive (and so in particular, $H^0(X,\mO_X(nD))\cong H^0(X_0,{\mO_X}_0(nD_0))$), we see from the local cohomology long exact sequence that we have an inclusion: $$H^1(X, \mO_X(-nD))\hookrightarrow H^1(X_0, \mO_{X_0}(-nD_0))$$ for all $n$. By our assumptions, $\alpha $ therefore defines a non-trivial element: $$\alpha_0\in H^1(X_0, \mO_{X_0}(-D_0)),$$ which belongs to the kernel of the morphism: $$H^1(X_0, \mO_{X_0}(-D_0))\rightarrow H^1(X_0, \mO_{X_0}(-pD_0))$$ induced by $F_0$. This kernel has a \emph{geometric} description on $X_0$ as follows:

 \subsubsection{$\alpha_{L}$-torsors}\label{alpha}
An invertable sheaf $ \mathcal{L}$ on a  scheme $X$ naturally defines a sheaf of groups under addition representable by the affine group scheme $L\defeq \Spec_X(\bigoplus_i \mathcal{L}^{-i}).$ Let $X$ be defined over a field $k$ of characteristic $p>0.$ Then there exists a homomorphism of sheaves of additive groups $\mathcal{L}\rightarrow \mathcal{L}^p$, defined by raising a local section to its $p$-power. \emph{This is a purely characteristic $p>0$ phenomenon,} since for any two local sections $s$ and $t$ of $\mathcal{L}$ we have $(s+t)^p=s^p+t^p$.\\
Relative to the Zariski site of $X$ the above morphism of sheaves is not in general a surjection. However, as a morphism of $X_{\fppf}$-sheaves it is, i.e., the corresponding morphism of sheaves of groups on the flat site of $X$ is surjective \cite[II.2.18]{Milne}. The kernel is a sheaf of groups on $X_{\fppf}$ which we denote by $\alpha_{\mathcal{L}}.$ By construction $\alpha_{\mathcal{L}}$ is representable by an affine group scheme $\alpha_L. $ In fact, we have $\alpha_{L}=\Spec_X( \bigoplus_{i=o}^{p-1} \mathcal{L}^{-i})$ since it is equal to the relative spectrum over $X$ of the cokernel of the $\mO_{X}$-algebra inclusion: $$\bigoplus_i \mathcal{L}^{-pi}\hookrightarrow \bigoplus_i \mathcal{L}^{-i}.$$
 \\
The short exact sequence of sheaves of groups (relative to the flat- topology on $X$) 
$$\begin{tikzcd} 0\arrow{r} &\alpha_{\mathcal{L}}\arrow{r} &{\mathcal{L}}\arrow{r} &{\mathcal{L}}^p\arrow{r} &0\end{tikzcd}$$
 induces a long exact sequence on cohomology \cite[4, Prop 4.5]{Milne}:
$$\dots \rightarrow H^0(X, \mathcal{L}^p)\rightarrow H^1_{\fppf}(X, \alpha _{\mathcal{L}})\rightarrow H^1(X, \mathcal{L})\rightarrow H^1(X, \mathcal{L}^p)\rightarrow \dots$$

Therefore, a non-trivial element $\alpha$ in the kernel of $H^1(X, \mathcal{L})\rightarrow H^1(X, \mathcal{L}^p)$  defines a non-trivial element $\alpha \in H^1_{\fppf}(X, \alpha _{\mathcal{L}}).$ This group has a geometric meaning, i.e., the first cohomology group of a group scheme $G$ over a scheme $X$ corresponds to a $G$-torsor $Y\rightarrow X$, see for example, \cite[Proposition 4.6]{Milne}. The torsors $Y$ were studied by Ekedahl:

\begin{prop}[{\cite[p 106-107]{Ekedahlcanmod}\cite[Theorem 2.11]{PatWal}}]\label{Ekedahl}Let $X$ be a normal variety over an algebraically closed field $k$ of characteristic $p>0$. Let $\mathcal{L}\in Pic(X)$. A non-trivial element of the kernel under the Frobenius action:
$H^1(X, \mathcal{L})\rightarrow H^1(X, \mathcal{L}^p),$ gives rise to a non-trivial $\alpha_\mathcal{L}$-torsor $\beta: Y \rightarrow X$. Locally over $ U=Spec(A)$ the $\alpha_L$- torsor $\beta^{-1}U \rightarrow U$ is given by $A\rightarrow A[x]/x^p-a$ for some element $a\in A$ which is not a $p$-power. Therefore, $\beta$ is purely inseparable of degree $p$. If $X$ is a $G_1$ and $S_2$ variety, then $Y$ is also a $G_1$ and $S_2$ variety which satisfies:
 $\omega_Y=\beta^{*}(\omega_X\otimes \mathcal{L}^{p-1})$. \end{prop}

Here $S_2$ denotes Serre's condition two and $G_1$ denotes Gorenstein in codimension one.

\begin{remark} Y is in general not normal. \end{remark}

 \subsection{Geometric construction for Weil-divisors}
Let $X$ be a projective normal surface over an algebraically closed field of charcteristic $p>0$ and let $D$ a $\Q$-Cartier ample Weil divisor on $X$ such that $H^1(X, \mO_X(-D))\neq 0$ and $H^1(X, \mO_X(-pD))=0.$ Let $X_0$ denote the smooth locus of $X$ and let $D_0=D\cap X_0$. Let $\alpha_0$ be a non-trivial element of $H^1_{\fppf}(X, \alpha _{\mO_{X_0}(D_0)})$ coming from a non-trivial element $\alpha \in H^1(X, \mO_X(D))$ as described in point \ref{alphasit}. Let $\pi_0 :Y_0\rightarrow X_0$ be the corresponding non-trivial $ \alpha _{\mO_{X_0}(D_0)}$-torsor. Then ${\pi_0}_*\mO_{Y_0}$ is a locally free $\mO_{X_0}$-algebra. %In fact, by the local description of $Y_0$, it is clear that $\pi_0 :Y_0\rightarrow X_0$ is finite and flat, therefore ${\pi_0}_*\mO_{Y_0}$ is a locally free $\mO_{X_0}$-module (see e.g., \cite{gortz}[Prop. 12.19]).  
Let $i:X_0\hookrightarrow X$ denote the inclusion of the regular locus. 
The multiplication on $X_0$ extends (because $i_*{\pi_0}_*\mO_{Y_0}$ is reflexive) to make $i_*{\pi_0}_*\mO_{Y_0}$ into a sheaf of $\mO_X$-algebras. We may therefore define $Y\defeq \Spec_X(i_*{\pi_0}_*\mO_{Y_0})$. Let $\pi: Y\rightarrow X$ denote the natural morphism.

\begin{lem}\label{torassD} In the situation above $\pi: Y\rightarrow X$ is a finite degree $p$ morphism and $Y$ is a projective $S_2$ and $G_1$ surface that satisfies:

$$K_Y=\pi^*(K_X -(p-1)D).$$ \end{lem}

\begin{proof}Since $\pi_*\mO_Y=i_*{\pi_0}_*\mO_{Y_0}$ is $S_2$ we have by \cite[Proposition 5.4] {KolMor} that $Y$ is $S_2. $ Since $X$ and $X_0$ agree in codimension one, $Y$ is $G_1.$ Since everything is $S_2$ and $G_1$ we may replace $X$ with $X_0$ by \cite[Theorem 1.12]{genS2G1}, therefore  $\pi^*(K_X)$ and $\pi^*(D)$ are well-defined and the formula $K_Y=\pi^*(K_X -(p-1)D)$ follows from Proposition \ref{Ekedahl}. \end{proof}

\begin{lem}\label{torassDnorm} With the notation as above, let $\eta :\overline{Y}\rightarrow Y$ be the normalization of $Y$ and let $\overline{\pi}:  \overline{Y}\rightarrow X$ denotes the induced morphism. There exists an effective $\Z$-divisor $\Delta\geq 0$ on $\overline{Y}$ such that $K_{\overline{Y}}+\Delta =\overline{\pi }^*(K_X -(p-1)D).$
\end{lem}

\begin{proof} Since everything is $S_2$ and $G_1$ we may replace $Y$ with a big open set and assume that $Y$ is Gorenstein and $\overline{Y}$ is regular. Affine locally we may assume that $\omega_Y\cong \mO_Y$ and  $\omega_X\cong \mO_X$. In this situation \cite[Lemma 2.14]{PatWal} proves that the natural map $\omega_{\overline{Y}}\rightarrow \eta^*\omega_{Y}$ has image equal to the conductor ideal, where $f^*\omega_Y$ is identified with $\mO_X$.  \end{proof}

\subsection{Bend and break}

By studying different properties of the variety $\overline{Y}$ appearing in Lemma \ref{torassDnorm} one hopes to arrive at a contradiction to the assumed non-vanishing, $H^1(X, \mO_X(-D))\neq 0$. This strategy was successfully employed by Ekedahl \cite{Ekedahlcanmod} and more recently in \cite{PatWal}. The main tool is the use of bend and break, together with the expression for the canonical divisor given by Lemma \ref{torassDnorm}. Together, this can be used in order to derive inequalities comparing intersection numbers on $X$ with the characteristic $p$ of the base field. We will repeatedly use the following:

\begin{thm}[{\cite[Theorem 5.8]{KolRat}}]\label{bendbreak}  Let $X$ be a projective variety over an algebraically closed field, $C$ a smooth, projective and irreducible curve, $f: C\rightarrow X$ a morphism and $M$ any nef $\mathbb{R}$-divisor. Assume that $X$ is smooth along $f(C)$ and $-K_X.C>0.$ Then for every $x\in f(C)$ there is a rational curve $\Gamma_x\subset X$ containing $x$ such that:
$$M.\Gamma_x\leq 2\dim(X)\frac{M.C}{-K_X.C} $$\end{thm}

\section{Kodaira vanishing for divisors of big volume}
We will use the following terminology: We will say that a curve $C$ on a surface $X$ is \emph{big and basepoint free} if the linear system $|C|$ defines a birational morphism. If $C$ is big and basepoint free, then for any finite number of points in $X$ there exists $C'\in |C|$ such that $C'$ avoids all those points.

 \begin{lem}\label{bigvollem}  Let $X$ be a projective normal surface. Suppose that there exists an effective $\Q$-divisor $E$ such that $-(K_X+E)$ is a $\Q$-Cartier ample divisor of Cartier index $I$. If $D$ is a big $\Q$-Cartier semiample $\Z$-divisor such that $D^2\geq I^2(K_X+E)^2$ then there exists an ample Cartier divisor $A$ and a big basepoint free curve $C$ on $X$  such that $D.C\geq A.C$. \end{lem}
 
 \begin{proof} Suppose, in order to arrive at a contradiction, that for every big basepoint free curve $C$ and every ample Cartier divisor $A$ of $X$ we have $D.C< A.C$. For $m$ divisible enough, $-m(K_X+E)$ is very ample. Therefore, for $A=-I(K_X+E)$ and $C\in |-m(K_X+E)|$ we find $-(K_X+E).D<I(K_X+E)^2$. Similarly, for $n$ divisible enough, $nD$ is basepoint free and we may consider a curve $C\in |nD|$ to find that $D^2<-I(K_X+E).D$. Putting this together, we find that $D
^2<I^2(K_X+E)^2$. 
\end{proof}

 \begin{thm}\label{vanishbig2}
 Let $X$ be a projective normal surface over an algebraically closed field $k$ of characteristic $p\geq 5$. Suppose that there exists an effective $\Q$-divisor $E$ such that $-(K_X+E)$ is a $\Q$-Cartier ample divisor of Cartier index $I$. If $D$ is a $\Q$-Cartier nef and big semiample $\Z$-divisor such that $D^2\geq I^2(K_X+E)^2$ then $H^1(X, \mO_X(-D))=0$.
 \end{thm}
 
 \begin{proof} Suppose that the theorem is not true. Let $D$ be a nef and big semiample  $\Z$-divisor on $X$  such $D^2\geq I^2(K_X+E)^2$ and assume that $H^1(X,\mO_X(-D))\neq 0$. By Corollary \ref{Sernefbigcor} there exists some $j\geq 1$ such that $H^1(X,\mO_X(-p^jD))= 0$. If we replace $D$ with $p^{j-1}D$ the inequality will be satisfied for $p^{j-1}D$. Suppose we can prove the theorem for $p^{j-1}D$ then $H^1(X,\mO_X(-p^{j-1}D))= 0$ and, hence, by descending induction on $j$, we have proven  that $H^1(X,\mO_X(-D))= 0$. Therefore, there is no loss of generality to assume that $H^1(X,\mO_X(-pD))=0$. \\
Let $\pi : Y\rightarrow X$ be a degree $p$ cover corresponding to a non-zero element $y\in H^1(X,\mO_X(-D))$ as in subsection \ref{alphasit}. By replacing $Y$ with its normalisation, we may assume that $Y$ is normal and that \begin{equation}\label{canonical}K_Y=\pi^*K_X +(1-p)\pi^*D -\Delta \end{equation} where $\Delta$ is effective (Lemma \ref{torassDnorm}). \\
By Lemma \ref{bigvollem} there exists a big basepoint free curve $C$ of $X$  and a ample Cartier divisor $A$ such that $D.C\geq A.C$. Let $\pi : Y\rightarrow X$ be as above. By replacing $C$ with a large enough multiple we see that there exists a curve $C_Y$ on $Y$ such that: \begin{itemize}
\item $C_Y$ is contained in the smooth locus of $Y$
\item $C_Y$ is not contained in a component of $\Delta$
\item $C=\pi_*C_Y$ for some curve $C$ on $X$ not contained in a component of $E$ and such that $D.C\geq A.C.$ 
\end{itemize}
By Equation \ref{canonical} we have $-K_Y.C_Y= -K_X.C + (p-1)D.C +\Delta. C_Y\geq -K_X.C$, where the last inequality comes from the assumption that $D$ is nef and that $C_Y$ is not contained in a component of $\Delta$. By assumption $-(K_X+E)$ is ample, hence $-K_X.C>E.C\geq 0$. Therefore $-K
_Y. C_Y>0$. By Theorem \ref{bendbreak} for $C_Y$ and $\pi^*A$ on $Y$, there exists a rational curve $\Gamma_x$ passing through a point $x\in C_Y$ such that: $$\pi^*A.\Gamma_x \leq 4 \frac{\pi^*A.C_Y}{-K_Y.C_Y} .$$ By Equation \ref{canonical}, we find: $$\left(\pi^*A.\Gamma_x\right)\left((-\pi^*K_X +(p-1)\pi^*D +\Delta).C_Y\right) \leq 4\pi^*A.C_Y.$$
However, since $A$ is an ample Cartier divisor and $\pi$ is finite $\pi^*A.\Gamma_x\geq 1$. Moreover, $-\pi^*K_X.C_Y=-K_X.C>0$ and $C_Y.\Delta \geq 0$.
By assumption, $\pi^*A.C_Y\leq \pi^*D.C_Y$, therefore:
 $$(p-1)\pi^*A.C_Y < \left(\pi^*A.\Gamma_x\right)\left((-\pi^*K_X +(p-1)\pi^*D +\Delta).C_Y \right).$$ 
Putting these inequalities together we find:  $$(p-1)\pi^*A.C_Y < 4 \pi^*A.C_Y$$
and, therefore, $p< 5$.

\end{proof}

Theorem \ref{vanishbig2} implies, together with results of  \cite{jiangbab}, our main theorem for divisors of big volume:

\bigvolcor*

\begin{proof} By the Basepoint free theorem  \cite[Theorem 1.3]{TanMMPsurf}, $D$ is semiample. Note that a klt surface of index $I$, necessarily, is $\frac{1}{I}$-klt. By \cite[Theorem 1.3]{jiangbab}, we therefore have that $(K_X+\Delta )
^2\leq \textrm{max}\{9, 2I +4+\frac{2}{I}\} .$ The result therefore follows from Theorem \ref{vanishbig2}.

\end{proof}

\begin{cor} Let $X$ be a klt del Pezzo surface over an algebraically closed field of characteristic $p\geq 5$, such that $-K_X$ is ample of Cartier index $I$. Let $D$ be a big and nef $\Z$-divisor of Cartier index $c$. Then $H^1(X, \mO_X(D))=0$ if either of the following holds: \begin{itemize}
\item $K_X^2\geq (c-1)^2D^2$ and $D$ is ample
\item $D^2\geq (I-1)^2K_X^2.$ 
\end{itemize}
\end{cor}
\begin{proof} We have that $D-K_X$ is ample and $H^1(X, \mO_X(K_X-D))$ is dual to $H^1(X, \mO_X(D))$. Assume that the claimed vanishing does not hold. Since the characteristic of the base field is greater than or equal to five, we may argue as in the proof of Proposition \ref{vanishbig2} to assume that for every big basepoint free curve $C$ and every ample Cartier divisor $A$ we have $(D-K_X).C< A.C.$ The first inequality comes from setting $A=cD$ and $C=-mK_X$ and $C=lD$ repectively, for large and divisible enough $l$ and $m$. The second inequality comes from setting $A=-IK_X$ and $C=-mK_X$ and $C=lD$ respectively, for large and divisible enough $l$ and $m$.

\end{proof}

\section{Kodaira vanishing in large characteristic}

In this section we prove Kodaira vanishing in large characteristic for log del Pezzo surfaces of bounded index. That bounded families of log del Pezzo surfaces over $\Spec(\Z)$ satisfy Kodaira vanishing in large characteristic (depending only on the family) has been proven by other methods in \cite{CTW}. Our first Proposition is quite general, it does not demand boundedness of the pair $(X, \Delta)$, it simply demands a very ample line bundle on $X$ of bounded degree and that $X$ is of Fano-type.

\begin{prop}\label{bound} Let $r>0$ be an integer. Let $(X, \Delta)$ be a log pair where $X$ is a projective normal surface over an algebraically closed field $k$ of characteristic $p\geq 4r+1$. Assume that $-(K_X+\Delta)$ is $\Q$-Cartier and ample. If there exists a very ample divisor $A$ on $X$ such that $A^2\leq r$ then Kodaira vanishing (even for big and nef $\Q$-Cartier $\Z$-divisors) holds on $X$.\end{prop}

\begin{proof} Suppose that $D$ is a big and nef $\Z$-divisor and suppose that\\ $H^1(X,\mO_X(-D))\neq 0$.  By replacing $D$ with $p^jD$ for some $j\geq 1$ we may assume that $H^1(X,\mO_X(-pD))= 0$ (Corollary \ref{Sernefbigcor}). By Ekedahl's construction we may assume that there exists a normal variety $Y$ and a purely inseparable degree $p$ morphism $\pi :Y\rightarrow X$ such that \begin{equation}\label{canonical2} K_Y=\pi^*K_X -(p-1)D-E \end{equation} for $E \geq 0$. Let $A$ be a very ample general curve on $X$ as above. There exists an integer $l>0$ such that $l\pi^*A$ is very ample on $Y$. Therefore there exists a curve $A_Y$ on $Y$ such that: \begin{itemize}
\item $A_Y$ is contained in the smooth locus of $Y$
\item $A_Y$ is not contained in a component of $E$
\item $A'=\pi_*A_Y$ for some $A'\in |nA|$ where $n>0$ is an integer.
\end{itemize}

By Equation \ref{canonical2} we have $$-K_Y.A_Y= n(-K_X.A + (p-1)D.A) + E. A_Y\geq -nK_X.A.$$ By assumption $-(K_X+\Delta)$ is ample, hence $-K_X.A>\Delta.A\geq 0$. Therefore $-K
_Y. A_Y>0$. By Theorem \ref{bendbreak}  for $A_Y$ and $\pi^*(A)$ on $Y$, there exists for every point $x\in A_Y$ a rational curve $\Gamma_x$ passing through $x$ such that $$\pi^*(A ).\Gamma_x \leq 4 \frac{\pi^*(A).A_Y}{-K_Y.A_Y} .$$ However, since $A$ is an ample Cartier divisor and $\pi$ is finite, $\pi^*A.\Gamma_x\geq 1$. By Equation \ref{canonical2} and the assumption $A^2\leq r$, we therefore have: $$(-\pi^*K_X +(p-1)\pi^*D + E)A_Y \leq 4rn.$$
Since $-K_X.A>0$ and $E.A_Y\geq 0$, we find that: $$(p-1)\pi^*D.A_Y < 4rn.$$
Since $\pi^*D.A_Y\geq n$, the result follows.

\end{proof}

\largecharI*

\begin{proof}By \cite{Wit15} we have that $(13-45I)IK_X$ is very ample. By \cite[Theorem 1.3]{jiangbab} we have that $K_X
^2\leq \textrm{max}\{9, 2I +4+\frac{2}{I}\} .$ The result therefore follows from Proposition \ref{bound} and Theorem \ref{thm:bigvolcor}.

\end{proof}

For the boundary version of the theorem above we need to use a common bound for the Cartier indices of $K_X$ and $K_X+\Delta$.

\begin{thm}\label{largecharIlog} Let $(X,\Delta)$ be a log del Pezzo surface over an algebraically closed field $k$ of characteristic $p>0$. Let $m\geq 2$ be such that $m(K_X+\Delta)$ and $mK_X$ are both Cartier. Let $D$ be a big and nef $\Z$-divisor on $X$. If $$p\geq 173056m^8-86528m^7 +16200m^5 +32400m^4 +16200m^3+1$$ then $H^1(X, \mO_X(-D))=0$. \end{thm}

\begin{proof} Let $H=13mK_X-45m^2(K_X+\Delta)$. According to \cite{Wit15}, $H$ is a very ample Cartier divisor. By \cite[Lemma 6.2]{Wit15}, we have inequalities $(K_X+\Delta).K_X\geq 0$ and $K_X^2 \leq 128m^5(2m-1)$. By \cite[Theorem 1.3]{jiangbab}, we have that $(K_X+\Delta )^2\leq 2m +4+\frac{2}{m} $ . Therefore: $$H^2\leq (13)^2128m^7(2m-1)+(45)^2(2m^5 +4m^4+2m^3).$$ We conclude by Proposition \ref{bound}.\end{proof}

\begin{remark}A klt log del Pezzo surface $(X, \Delta )$ where $(K_X+\Delta)$ has Cartier index $I$ is necessarily $\epsilon$-klt for $\epsilon =\frac{1}{I}$. As in the proof below, the Cartier index of $K_X$ can therefore be bounded in terms of the Cartier index of $(K_X+\Delta)$.\end{remark}

\vaneps*

\begin{proof}  By assumption $X$ is $\epsilon$-klt, therefore the $\Q$-factorial index $q$ at any point $x\in X$ satisfies $$q\leq 2\left(\frac{2}{\epsilon}\right)^{\frac{128}{\epsilon^5}}$$ by \cite[Prop 6.1]{Wit15}. Moreover, by the same proposition, the Picard rank of a minimal resolution of $X$ is bounded by $\frac{128}{\epsilon^5}$. Hence, for every $\Z$-divisor $D$, there exists some $c\leq 2^{\frac{128}{\epsilon^5}}\left(\frac{2}{\epsilon}\right)^{\frac{(128)^2}{\epsilon^{25}}},$ depending on $D$, such that $cD$ is Cartier. \\
Suppose that $H^1(X,\mO_X(-D))\neq 0$. By replacing $D$ with $p^{k-1}D$ for some $k>1$, we may assume that $H^1(X,\mO_X(-pD))=0$. We may, therefore, assume that there exists a normal variety $Y$ and a degree $p$ inseparable morphism $Y\rightarrow X$ such that $K_Y= \pi^*(K_X+(1-p)D)-E$, for some $E \geq 0$. Let $C_Y$ be a general complete intersection curve on $Y$, not contained in the support of $E$, such that $\pi_*C_Y=C$ is not contained in the support of $\Delta$. Then $-K_Y.C_Y>0$. By Theorem \ref{bendbreak} applied to $C_Y$ and $$\pi^*(-(K_X+\Delta )+(p-1)D)\equiv -(K_Y+E)-\pi^*(\Delta ),$$ there exists a rational curve $\Gamma$ through a point on $ C_Y$ such that: $$\pi^*(-(K_X+\Delta)+(p-1)D).\Gamma \leq 4\frac{ (-(K_Y+E)-\pi^*(\Delta )).C_Y}{-K_Y.C_Y}.$$
Since $-\pi^*(K_X+\Delta ).\Gamma >0$, $E.C_Y\geq 0$ and $\pi^*(\Delta ).C_Y\geq 0$, we find that $$(p-1)D.\Gamma < 4.$$ We have that $D.\Gamma \geq \frac{1}{c}.$
 Therefore: $$p< 4c+1\leq 
 2^{\frac{128}{\epsilon^5}+2}\left(\frac{2}{\epsilon}\right)^{\frac{(128)^2}{\epsilon^{25}}}
 +1.$$

\end{proof}

\begin{remark}\label{Kodairavanforcartier} From the proof of Theorem \ref{van eps} we see that on a log del Pezzo surface $(X,\Delta)$ over an algebraically closed field $k$ of characteristic $p \geq 4c+1$ we have $H^1(X,\mathcal{O}_X(-D))=0$, for every ample $\Z$-divisor $D$ of Cartier index $\leq c$.\end{remark}

\section{Kodaira vanishing in char $p>5$ using the liftability of log del Pezzo surfaces of Picard rank one}

\begin{prop}\label{MFS curve} Let $X$ be a normal surface birational to a surface $Y$ via a birational morphism $g:X\rightarrow Y$, such that $Y$ admits a fibration $f:Y\rightarrow B$ to a curve $B$ with general fiber $F$ isomorphic to $\mathbb{P}^1$. Let $D$ be a $\Q$-Cartier big and nef $\Z$-divisor on $X$. Then $H^1(X, \mO_X(-D))=0$. \end{prop}

\begin{proof}According to Corollary \ref{Sernefbigcor} by replacing $D$ with $p^mD$, we may assume that $H^1(X, \mO_X(-pD))=0.$ Suppose that
$H^1(X, \mO_X(-D))\neq 0$. We apply Ekedahl's construction \ref{alphasit} to get a degree $p$ finite morphism $\pi: Z\rightarrow X$, where $Z$ is normal, and $$K_Z+E= \pi^*K_X+(1-p)\pi^*D$$ for an effective divisor $E\geq 0.$ We illustrate the situation with a diagram:
$$\begin{tikzcd} Z\arrow{d}{\pi}\\
X\arrow{r}{g} &Y\arrow{d}{f} \\ 
&B \end{tikzcd}$$ Let $F$ be a fiber of $f$ not containing any point which is the image of a curve contracted by $g$. Then $F$ can be identified with a fiber of $f\circ g:X\rightarrow B$. By abuse of notation we denote a general fiber of $f\circ g$ by $F$. A general fiber of the composition $f\circ g$ avoids the exceptional locus of $g$ since the curves contracted by $g$ map to a finite number of points on $Y$. We compute: $$K_Z. \pi^*F=\pi^*((K_X-(p-1)D).F)- E.\pi^*F.$$ The morphism $\pi$ is finite of order $p$, therefore, we have: $$K_Z.\pi^*F=p((K_X-(p-1)D).F-E.  \pi^*F.$$ Since two general fibers $F$ do not intersect we have $F^2=0$ and therefore $K_Y.F=-2$. Since $D$ is big and $F$ is general, we have $D.F> 0$. \\
 
 Moreover, since $E$ is effective, we see (upon varying $F$) that the intersection number $E. \pi^*F\geq 0$. We conclude that:
 $$K_Z.  \pi^*F\leq p(K_X-(p-1)D). F=$$
 $$p(deg(K_{ F})-(p-1)D. F)=$$
 $$-2p -p(p-1)D. F.$$
 
 Denote by $G=\pi^*F$. Let $F'$ denote the reduction of a general fiber of $f\circ g\circ \pi$. 
 
 If $G$ is reduced, then $G=F'$, otherwise $G=pF'.$\\
 In either case, we have: $$-2p\leq K_Z.G\leq -2p -p(p-1)D.F,$$
 but $-p(p-1)D.F<0.$ This is a contradiction.
  \end{proof}

\begin{definition}[{\cite[Definition 8.11]{EsnVieh}}] Let $k$ be a perfect field of positive characteristic and let $X$ be a scheme smooth over $\Spec(k)$. Let $D=\sum_{i}^nD_i$ be a reduced simple normal crossing divisor on $X$.
A lifting of $D=\sum_{i=1}^nD_i\subset X$ to $\W_2(k)$ consits of a scheme $X'$ and subschemes $D'_i$ all defined and flat over $\W_2(k)$ such that: 
\begin{center}$X=X'\times_{\Spec(\W_2(k))}\Spec(k)$ and $D_i=D_i'\times_{\Spec(\W_2(k))}\Spec(k)$,\end{center} for $1\leq i\leq n$.\end{definition}

For a pair $(X,D)$ consisting of a smooth surface and a reduced simple normal crossing divisor as above, \cite{CTW} defines the notion of \emph{liftability to characteristic zero over a smooth base}. 
We refer the reader to  \cite[Definition 2.15]{CTW} for the definition. If the pair $(X, D)$ lifts to characteristic zero over a smooth base then the pair $(X, D)$ lifts to $\W_2(k)$ \cite [Remark 2.16]{CTW}.

\begin{thm}[{\cite[Theorem 7.2]{JustinLacini}}]\label{JL} Let $S$ be a log del Pezzo surface of Picard rank one over an algebraically closed field of characteristic $p>5$. Then there exists a log resolution $\mu: V\rightarrow S$ such that $(V,\Exc(\mu))$ lifts to characteristic zero over a smooth base. %In particular, $(V,Exc(\mu))$ lifts to $W_2(k)$.
 \end{thm}

We believe the following lemma to be well-known to experts, however since we did not find a reference we include it here.

\begin{lem}\label{liftblowup} Let $X$ be a smooth variety and $D=\sum_{i=1}^nD_i$ be a reduced simple normal crossing divisor. Suppose that $(X, D)$ lifts to $\W_2(k)$. Let $x\in X$ be a closed point. If $\pi: Y \rightarrow X$ is the blow up of $X$ at $x$ then $(Y, Exc(\pi)+\pi_*^{-1}D)$ lifts to $\W_2(k)$.\end{lem}
\begin{proof} Let $(X',D')$ be a lift of $(X, D)$ to $\W_2(k)$. Let $x\in X$ be a closed point. By formal smoothness there exists a lift $x'\in X'$ of $x$. We claim the following:
\begin{claim} If $J\subset \{1,2,\dots , n\}$ is an index set such that $x\in \bigcap_{i\in J}D_i$ but $x\notin D_k$ for $k\notin J$, then, there exists a lift $x'\in X'$, such that $x'\in \bigcap_{i\in J}D'_i$, but $x'\notin D'_k$ for $k\notin J$.\end{claim}

 By \cite[Lemma 8.13 d)]{EsnVieh}, we may assume that we have a diagram with all squares being Cartesian:
$$ \begin{tikzcd} U\arrow{rrrrr}\arrow{rrd}{e}\arrow{dd}&&&&&U'\arrow{lld}[swap]{e'}\arrow{dd}\\
 &&\mathbb{A}_k\arrow{r}\arrow{lld}&\mathbb{A}_{\W_2(k)}\arrow{rrd}\\
 \Spec(k) \arrow{rrrrr}&&&&&\Spec(\W_2(k))\end{tikzcd}$$
 
where $\mathbb{A}_k=\Spec(k[t_1, t_2])$, $\mathbb{A}_{\W_2(k)}=\Spec(\W_2(k)[t_1, t_2])$, $U$ and $U'$ are open subsets of $X$ and $X'$ respectively and the morphisms $e$ and $e'$ are \'etale. By \cite[Lemma 8.14 e)]{EsnVieh}, we may assume that we have chosen local parameters:
 
  \begin{center} $\phi_i=e^*(t_i)\in \mO_U$ and $\phi'={e'}^*(t_i)\in \mO_{U'}$ \end{center}

such that $x=\V((\phi_1,\phi_2))$ and $x'=\V(({\phi'}_1,{\phi'}_2))$ and such that there exists a subset $I$ of $\{\emptyset , 1, 2\}$, such that $D\cap U=\V(\prod_{i\in I}\phi_i)$ and $D'\cap U'=\V(\prod_{i\in I}{\phi'}_i)$. This proves the claim. \\
 
 Let $\pi': Y'\rightarrow X'$ be the blow up of $x'=\V(({\phi'}_1,{\phi'}_2)),$ as above. We now prove that $(Y', {\pi'}^{-1}_*D'+\Exc(\pi'))$ is a lift of $(Y, \Exc(\pi)+\pi_*^{-1}D)$ to $\W_2(k)$. To this end, let $S$ be equal to $k$ or $\W_2(k)$ and $\mathbb{A}_S=\Spec(S[t_1, t_2])$ and $e:U\rightarrow \mathbb{A}^2_S$ an \'etale morphism. Then by the commutativity of blowing up with flat base change \cite[\href{https://stacks.math.columbia.edu/tag/085S}{Lemma 085S}]{stacks-project} the following natural diagram is Cartesian:
 
 $$\begin{tikzcd} \Bl_{V((e^*(t_1),e^*(t_2)))}(U)\arrow{d}\arrow{r}&\Bl_{V((t_1,t_2))}(\mathbb{A}^2_S)\arrow{d}\\
 U\arrow{r}{e}&\mathbb{A}^2_S\end{tikzcd}$$
 
 Moreover, the strict transform of $\V(\prod_{i\in I}e^*(t_i))$ is the base change to $U$ of the strict transform of $\V(\prod_{i\in I}t_i)$ for any subset $I$ of $\{\emptyset , 1, 2\}$. Therefore:
 $$\Bl_{x'}(U')=\Bl_{\V(t_1,t_2)}(\mathbb{A}^2_{\W_2(k)})\times_ {\mathbb{A}^2_{\W_2(k)}}U'$$ and $$\Bl_{x}(U)=\Bl_{\V(t_1,t_2)}(\mathbb{A}^2_k)\times_ {\mathbb{A}^2_k}U$$
 and the strict transform of the subschemes defined by the local parameters correspond in the adequate way. \\
We are therefore reduced to considering the following situation:
 $$\begin{tikzcd} \Bl_{\V(t_1, t_2)}\mathbb{A}^2_k \arrow{r}\arrow{d} & \Bl_{V(t_1, t_2)}\mathbb{A}^2_{\W_2(k)}\arrow{d}\\
 \mathbb{A}^2_k\arrow{r}\arrow{d} &\mathbb{A}^2_{\W_2(k)}\arrow{d} \\
 \Spec(k)\arrow{r} &  \Spec(\W_2(k))\end{tikzcd}$$
 
 The top square, and therefore the whole diagram, is clearly Cartesian. In this situation it is evident that $\Bl_{\V(t_1, t_2)}\mathbb{A}^2_{\W_2(k)}$ is smooth over $W_2(k)$ and that the strict transform of each of the coordinate axes, as well as the exceptional divisor of the blow up, are flat and hence smooth over $\W_2(k)$. Since the base change of an \'etale morphism is \'etale this shows that $Y', {\pi'}^{-1}_*D'_i$, for $1\leq i \leq n$, and $\Exc(\pi')$ are all flat over $\W_2(k)$. Consequently $(Y', {\pi'}^{-1}_*D'+\Exc(\pi'))$ is a lift of $(Y, \Exc(\pi)+\pi_*^{-1}D)$ to $\W_2(k)$.

 \end{proof}

\ratvannef*

\begin{proof} Let $D$ be an ample $\Q$-Cartier $\Z$-divisor. It is sufficient to prove that $0=H^1(X,\mO_X(-D))=H^1(X,\mO_X(K_X+D))$. We run a $K_X$-MMP to get a birational morphism $g: X\rightarrow Y$ where either $f:Y\rightarrow B$ is a MFS onto a curve $B$ or $Y$ is a klt del Pezzo surface of Picard rank one. 

\begin{enumerate} \item If $f:Y\rightarrow B$ is a $K_X$-MFS onto a curve, then the result follows from Proposition \ref{MFS curve} . \\

\item If there exists a birational morphism $X\rightarrow Y$ where $Y$ is a log del Pezzo surface of Picard rank one then by Theorem \ref{JL} there exists a log resolution $\mu: (V, \Exc(\mu))\rightarrow Y$ such that $(V, \Exc(\mu))$ lifts to characteristic zero over a smooth base where $\Exc(\mu)$ denotes the reduced divisor supported at the exceptional locus of $\mu$. There is an induced rational map $\mu':V\dashedrightarrow X$.  Let $\pi:V'\rightarrow V$ be a resolution of the indeterminacy locus of $\mu'$. The morphism 
$\pi:V'\rightarrow V$ can be facorised as a composition of blowups at points. By Lemma \ref{liftblowup}, $(V', \pi_*^{-1}\Exc(\mu)+ \Exc(\pi))$ lifts to $\W_2(k)$. The diagram below illustrates the situation:
$$\begin{tikzcd} X\arrow{d}{g}&(V',\pi_*^{-1}\Exc(\mu)+\Exc(\pi))\arrow{l}{\tilde{\mu }}\arrow{d}{\pi}\\
Y&(V,\Exc(\mu))\arrow{l}{\mu}\end{tikzcd}$$
We have that the support of $\Exc(\tilde{\mu})$ is contained in the support of $\pi_*^{-1}\Exc(\mu)+\Exc(\pi)$. Therefore $\tilde{\mu}:(V',\Exc(\tilde{\mu}))\rightarrow X$ is a log-resolution that lifts to $\W_2(k)$. By \cite[Lemma 6.1]{CTW} the result follows. 

\end{enumerate}\end{proof}

\bibliography{phdbib}{}
\bibliographystyle{alpha}

\end{document}